\documentclass[12pt,leqno]{amsart}
\usepackage{amssymb,amsmath,amsthm,bm}
\usepackage{graphicx}
\usepackage[dvipsnames]{xcolor}
\usepackage{color}
\usepackage[textsize=tiny]{todonotes}
\usepackage[normalem]{ulem}
\usepackage[bookmarksopen,bookmarksdepth=3,colorlinks,citecolor=red,pagebackref,hypertexnames=true]{hyperref}
\usepackage[msc-links,nobysame,non-sorted-cites, initials]{amsrefs}
\usepackage[inline]{enumitem}
\usepackage{mathtools}
\usepackage{yfonts}
\mathtoolsset{showonlyrefs}

\definecolor{indigo}{rgb}{0.4,0,0.9}
\usepackage{lipsum}
\usepackage{mathrsfs}


\usepackage{txfonts}
\usepackage[T1]{fontenc}

\newcommand{\norm}[1]{\left\Vert #1 \right\Vert}
  
\newcommand{\avg}[1]{\left\langle #1 \right\rangle}

\newcommand{\PM} {\mathsf{P}}

\newcommand{\bb} {{\mathfrak{b}}}

\newcommand{\D}{\mathcal D}

\newcommand{\Sy}{\mathsf{Sy}}
\newcommand{\Tr}{\mathsf{Tr}}

\newcommand{\ip}[2]{\left\langle {#1},{#2} \right\rangle}
\newcommand{\abs}[1]{\left\vert {#1}\right\vert}

\newcommand{\R}{\mathbb R}

\newcommand{\C}{\mathcal C}

\newcommand{\ep}{\varepsilon}

\DeclareMathOperator{\supp}{{supp}}

\newcounter{thms}

\newcounter{other}
\numberwithin{other}{section}

\newtheorem{proposition}[other]{Proposition}
\newtheorem{theorem}[thms]{Theorem}
\newtheorem*{theorem*}{Theorem}
\newtheorem*{proposition*}{Proposition}
\newtheorem{cor}{Corollary}
\newtheorem*{corollary*}{Corollary}
\numberwithin{cor}{thms}
\newtheorem{lemma}[other]{Lemma}
\theoremstyle{definition}
\newtheorem{remark}[other]{Remark}
\newtheorem{definition}[other]{Definition}

\numberwithin{equation}{section}

\author[F. Di Plinio]{Francesco Di Plinio}
\thanks{F.\ Di Plinio is partially supported by the FRA 2022 Program of University of Napoli Federico II, project ReSinAPAS -
Regularity and Singularity in Analysis, PDEs, and Applied Sciences. }

\address[F. Di Plinio]{Dipartimento di Matematica e Applicazioni, Universit\`a di Napoli \\ \newline \indent Via Cintia, Monte S.\ Angelo 80126 Napoli, Italy}
\email{\href{mailto:francesco.diplinio@unina.it}{\textnormal{francesco.diplinio@unina.it}}}

\author[A. W. Green]{A. Walton Green}
\thanks{A. W. Green's research supported by NSF grant NSF-DMS-2202813.}
\author[B. D. Wick]{Brett D. Wick}
\thanks{B. D. Wick's research partially supported in part by NSF grants NSF-DMS-2054863, NSF-DMS-2349868 as well as ARC DP 220100285.}

\address[A. W. Green, B. D. Wick]{Department of Mathematics, Washington University in Saint Louis\\ \newline \indent 1 Brookings Drive, Saint Louis, Mo 63130, USA}
\email{\href{mailto:awgreen@wustl.edu}{\textnormal{awgreen@wustl.edu}}, \href{mailto:bwick@wustl.edu}{\textnormal{bwick@wustl.edu}}}

\subjclass[2010]{42B20}
\keywords{Wavelet representation theorem,  paraproducts, Triebel-Lizorkin norms, sparse domination, multilinear Calder\'on-Zygmund theory, Sobolev spaces.}

\begin{document}

\title[Multilinear paraproducts]{Multilinear paraproducts on Sobolev spaces}

\begin{abstract} Paraproducts are a special subclass of the multilinear Calder\'on-Zygmund operators, and their Lebesgue space estimates in the full multilinear range are characterized by the $\mathrm{BMO}$ norm of the symbol. In this note, we characterize the Sobolev space boundedness properties of multilinear paraproducts in terms of  a suitable family of Triebel-Lizorkin type norms of the symbol. Coupled with a suitable  wavelet representation theorem, this characterization leads to a new family of Sobolev space $T(1)$-type theorems for multilinear Calder\'on-Zygmund operators. 
\end{abstract}	
\maketitle

\mathtoolsset{showonlyrefs}

\section{Introduction} 
 
Multilinear paraproducts are a special class of multilinear Calder\'on-Zygmund operators. The latter class plays a pivotal role in e.g.\ nonlinear partial differential equations, see \cites{KP,LS} for two well known examples, and  its  systematic study began with the works of Coifman-Meyer \cite{CM}, Kenig-Stein \cite{KS}, and Grafakos-Torres \cite{GT}. While  the mapping  properties of  \emph{linear} paraproducts are well understood in the Lebesgue setting, a sharp characterization of their inhomogeneous Sobolev space behavior was not available in past literature before \cite{DGW3} by these authors. This  note extends the recent results on paraproducts from \cite{DGW3} to the \emph{multilinear case}, and exploits this extension to produce   a new family of Sobolev space $T(1)$-type theorems for multilinear Calder\'on-Zygmund operators.

 We turn to a summary of the main results. The paraproduct operators in question are the $m$-linear operators given by
\begin{equation*}
 \Pi_{\bb}(f_1\ldots,f_m)(x) \coloneqq \sum_{ Q\in \D}  |Q|  b_Q  \zeta_{Q}(f_1,\ldots,f_m) \beta_Q (x), \qquad x\in \mathbb R^d 
 \end{equation*} 
 where \[\bb=\{b_{Q}: Q \in \D\}\] is a sequence of complex numbers indexed by the family of dyadic cubes $\mathcal D$ of $\R^d$,  and where, loosely speaking, $\beta_Q$ is an $L^1$-normalized cancellative wavelet adapted to the cube $Q$, while $\zeta_Q $  is an $m$-linear, non-cancellative averaging form also adapted to $Q$. Rigorous definitions are given in \S\ref{ss:wpf} below.
Identifying  the sequence $\bb$ with the function
\begin{equation*}
\bb\coloneqq \sum_{Q \in \D} |Q|b_Q \varphi_Q
\end{equation*} where $\{\sqrt{|Q|}\varphi_Q:Q\in \mathcal D\}$ is a suitable wavelet basis of $L^2(\R^d)$, the Lebesgue space theory of $m$-linear paraproducts, at least in the open range, may be summarized by the characterization
 \[\Pi_{\bb}:L^{p_1}(\R^d)\times\cdots\times L^{p_m}(\R^d)\to L^p(\R^d), \qquad \frac{1}{p}=\sum_{k=1}^{m}\frac{1}{p_k}, \quad 0<p<\infty, \, 1<p_1,\ldots, p_k\leq \infty\] if and only if $\bb\in \mathrm{BMO}(\R^d)$. 
This equivalence can be made quantitative, in the sense  the operator norm of the multilinear paraproduct is comparable to the $\mathrm{BMO}(\R^d)$ norm of the symbol $\bb$. See \eqref{e:lebbase} below.

 Theorem \ref{paraproduct-main}, Section \ref{sec:pp} obtains a Sobolev space analogue of the above characterization,  requiring natural sharp or near sharp, and in general much weaker conditions than $\mathrm{BMO}(\R^d)$ membership of $\bb$. These conditions are quantified upon certain suitably defined Triebel-Lizorkin type norms with predecessors in the literature, see e.g.\ \cite{YSY}.  
Following the approach of \cites{DGW1,DGW2,DGW3}, a sufficiently smooth multilinear Calder\;on-Zygmund operator may be decomposed  into a finite  sum of  purely cancellative multilinear Calder\'on-Zygmund operators and multilinear paraproduct operators $\Pi_{\bb}$ for appropriate symbols $\bb$ that are connected to the testing conditions of $T$ on monomials.  The reader can see this decomposition in Theorem \ref{thm:rep}.  Combining this representation with Theorem \ref{paraproduct-main} leads to  Corollary \ref{c:sob}, which is a testing type result, in the vein of David-Journ\'e \cite{DJ}, for Sobolev space boundedness of multilinear Calder\'on-Zygmund operators. 

For the sake of simplicity, we restricted ourselves to unweighted Sobolev estimates on all of  $\R^d$. Broader reaching generalizations, such as  testing-type theorems on weighted Sobolev spaces, or representation results covering the domain setting such as in \cites{DGW3,PT} may be considered in the multilinear setting as well. The statements of these cases are left to the interested reader.

\subsection*{Notational conventions} This article studies $n$-linear, $n\geq 1$, singular integral forms acting on tuples of functions on the ambient Euclidean space $\R^d$, $d\geq 1.$ The simplest example is the integral form associated to $\varphi \in L^1(\R^{dn})$ and acting on tuples $(f_1,\ldots, f_n) \in \big(L^{\infty}(\R^d)\big)^n $  \begin{equation}
	\label{e:intform}
\varphi(f_1,\ldots, f_n) \coloneqq \int\displaylimits_{\bigtimes_{j=1}^n \R^d} \varphi (x_1,\ldots, x_n) \prod_{j=1}^n f_j(x_j) \mathrm{d} x_j.
\end{equation}
A  cube $Q$ of $\R^d$ is the Cartesian product of left closed, right open intervals of equal length $\ell(Q)$ and its center is indicated by $c(Q)$.
 The long distance between any two cubes $Q,S\subset \R^d$ is defined as	\[ {\mathfrak d}(Q,S)=\max\{|c(Q)-c(S)|,\ell(Q),\ell(S)\}.\]
 The cubes of $\R^d$ also parametrize the   linear transformations 
\[\begin{split}
&\Sy^{p}_Q f(x)\coloneqq \frac{1}{\ell(Q)^{\frac{nd}{p}}} f\left(\frac{ x-\big(c(Q), \ldots, c(Q)\big)}{\ell(Q)}  \right), \quad x=(x_1,\ldots, x_n)  \in \R^{dn}
\end{split}
\]
with $ 1\leq p \leq \infty.$
Our wavelet decomposition is parametrized by the standard dyadic grid $\mathcal D$ on $\R^d$, and we make use of the notation
\[
\mathcal D(Q)\coloneqq\{R\in \mathcal D: R\subset Q\},  \qquad
Q\in \mathcal D.\] 
The  local  norm of $f\in L^1_{\mathrm{loc}}(\R^d)$ on the cube $Q\subset \R^d$ is denoted by 
\[
\| f\|_{L^p(Q)} \coloneqq
 \|\mathbf{1}_Q f\|_{L^p(\R^d)}, \qquad 
 \langle f \rangle_{p,Q} \coloneqq  |Q|^{-\frac1p} \| f\|_{L^p(Q)},   \quad 0\leq p<\infty.
\]
As customary within the subject, the constants implied by the almost inequality sign may vary at each occurrence and possibly depending on the parameters relevant to each inequality such as e.g.\ exponent tuples, degree of linearity and ambient Euclidean dimension.
\subsubsection*{Conflict of interest statement} On behalf of all authors, the corresponding author states that there is no conflict of interest.

\section{Wavelet resolution, wavelet forms}\label{sec:wavelets}
A fundamental result of Daubechies \cites{D1,D2}, yields a smooth, compactly supported orthonormal basis of $L^2(\mathbb R^d)$ subordinated to the multiresolution $\mathcal D$. We introduce the precise  statement together with our notation for wavelet classes. {{For $n\geq 1$, $0 \le \delta < 1$ and $\rho \ge 0$,   let  \[	\|\varphi\|_{n,\delta,\rho} := \sup_{x \in \R^{nd}} \left(1+ |x|\right)^{dn+\rho}  \left[|\varphi(x)| + \sup_{0 \le |h| \le 1} \dfrac{|\varphi(x+h)-\varphi(x)|}{|h|^{\delta}} \right]. \] 
 For each cube $R$ of $\R^d$, define the $L^1$-normalized class 
	\[ \Phi^{\sigma}_{n} (R)\coloneqq  \left\{  \Sy_R^1\varphi: \varphi \in W^{\lfloor \sigma \rfloor,\infty}(\R^{dn}): \sup_{0 \le |\alpha| \le \lceil \sigma\rceil -1 }\left\|\partial^\alpha \varphi\right\|_{n,\{\sigma\},\sigma} \leq 1 \right\}
	   \]
	   where  $\sigma > 0$ is a smoothness parameter with fractional part $\{\sigma\}=\sigma - \lceil \sigma \rceil+1$, as well as the cancellative subclass 
\begin{equation} 
\label{e:cancelcond}
\Psi^{\sigma}_{n} (R)\coloneqq  \left\{  \varphi\in \Phi^{\sigma}_n (R):\int\displaylimits_{\R^d} x_1^\alpha \varphi(x_1,  \ldots,x_n)\,  \mathrm{d}x_1 =0 \quad  \forall   0 \le |\alpha| \le \lceil \sigma\rceil -1 \right\}.
\end{equation}
In particular, the integral in \eqref{e:cancelcond} are absolutely convergent, as $\{\sigma\}>0$.
{Notice here that we are breaking the symmetry between the $x_1$ and $(x_2,\ldots,x_n)$ variables. We might have defined  a more general family of classes requiring vanishing moments for a subset of the variables $\{1,\ldots, n\}$ of cardinality $\geq 1$, but our needs will be limited to one single cancellative variable at a time.}
The additional superscript  $\Subset$ stands for  the subset of the corresponding class of functions having compact support in $\bigtimes_{j=1}^n \mathsf{w}R$, where $\mathsf{w} \geq 1$ is a dilation parameter specified in the statement of Proposition \ref{prop:triebel}. To wit, 
$\Phi^{\sigma,\Subset}_n (R)\coloneqq\{\varphi\in \Phi^{\sigma}_{n} (R): \mathrm{supp} \,\varphi \Subset \bigtimes_{j=1}^n\mathsf{w} R \}$. The subscript $n$ is omitted when $n=1$.
}

The basic starting point of our analysis is the following form of the wavelet resolution theorem by Daubechies.
\begin{proposition}\label{prop:triebel}  Let $k\geq 1$ be a fixed integer. Then there exists a positive constant $c$ and an $L^1$-normalized family
\begin{equation}
\label{e:triebelref}  \mathfrak F_k \coloneqq
\left\{ \chi_{Q}\in c\Phi^{k+1,\Subset}({Q}),\, \varphi_{Q}\in c\Psi^{k+1,\Subset}({Q}): {Q}\in \mathcal Q \right\}
\end{equation}
with the property that \begin{equation}\label{e:basis}
\mathfrak B^2 \coloneqq
\left\{\sqrt{|{Q}|}\varphi_{Q} : {Q}\in \mathcal D \right\} \end{equation} is an orthonormal basis of $L^2(\R^d)$ and for each $\ell \in \mathbb Z$ and $f \in L^2(\R^d) $, 
	\begin{equation}\label{e:high-low} \sum_{\substack{ Q \in \mathcal D \\ \ell(Q) > 2^\ell}} |Q| \varphi_Q(f) \varphi_Q  =	 \sum_{\substack{ Q \in \mathcal D \\ \ell(Q) = 2^\ell}} |Q| \chi_Q(f) \chi_Q.
	\end{equation} 
\end{proposition}
The above orthonormal expansion  will be taken advantage of in the representation of suitable Calder\'on-Zygmund forms. To ensure unconditional convergence throughout our formulas, it is convenient to work with the approximating  classes
\[
\mathcal {W}_0\coloneqq \mathrm{span}\, \mathfrak B^2, \qquad 
\mathcal W=\left\{f\in L^2(\R^d):  \sup_{Q\in \mathcal D}  \ell(Q)^{-5kd} \sqrt{|Q|}\left|  \varphi_Q(f)\right| <\infty\right\}
\]
whose explicit dependence on $k$ in the notation is omitted.
We omit the easy argument showing density of $\mathcal W$ in $W^{m,p}(\R^d)$ for all $0\leq m\leq k$ and $1<p<\infty$.

\subsection{Wavelet and paraproduct forms} \label{ss:wpf}
Our approach to multilinear singular integral forms is to expand them into simpler model forms of the same type, which we refer to as \textit{wavelet forms}. 
\begin{definition}[Wavelet forms] \label{def:forms}
Let $m\geq 1$ and $\sigma>0$. To collections
 \begin{equation} \label{e:collwave}
 \{\phi_Q \in \Psi^{\sigma,\Subset} (Q) : Q \in \D \}, \quad \{\psi_Q \in \Psi^{\sigma}_m (Q) : Q \in \D \}
 \end{equation} associate the $(m+1)$-linear \textit{wavelet form}
 \begin{equation}
\label{e:triebel2}
 \Lambda(f,f_1\ldots, f_{m}) \coloneqq \sum_{Q \in \D} |Q| \phi_Q(f)\psi_Q(f_1,\ldots, f_{m}).
 \end{equation}
 The savvy reader will realize that we have once again broken the symmetry and reduced the generality of the forms we consider, consistently with the observation made after \eqref{e:cancelcond}. To wit, we might have considered the more general class of forms generated by linear combinations of elements of $\Phi^\sigma_{m+1}(Q)$ with at least two cancellative entries, of which $\phi_Q \otimes \psi_Q $ appearing in \eqref{e:triebel2} is a special case. However, the generality of \eqref{e:triebel2} is sufficient for our purposes.

Each  wavelet form has $(m+1)$ adjoint     $m$-linear Calder\'on-Zygmund singular integral operators on $\R^d$, e.g. as defined in \cite{DGW1}*{Def.\ 3.2} and thus admit $L^p$ and sparse domination estimates. By virtue of the equivalence expounded in  \cite{CDPO1}, we may formulate the latter as domination by the  $n$-linear maximal operator
\[
\mathrm{M}_{\vec p} \left(f_1,\ldots, f_n\right) (x) \coloneqq \sup_{Q \subset \R^d \, \textrm{cube}} \mathbf{1}_Q(x) \prod_{j=1}^n \langle f_j \rangle_{p_j, Q}, \qquad x\in \mathbb R^d
\]
associated to a tuple of exponents $\vec p =(p_1,\ldots, p_n)\in (0,\infty)^n$. When $\vec p =(1,\ldots, 1) $,   the  subscript is omitted.
\begin{proposition} \label{p:sparseforw} Let $\Lambda$ be an $(m+1)$-linear wavelet form. Then
\[
\left| \Lambda(f,f_1\ldots, f_{m}) \right| \lesssim \min\left\{\|f\|_{\mathrm{BMO}(\R^d)} \left\|\mathrm{M}  \left(f_1,\ldots, f_m\right) \right\|_{1}, \left\|\mathrm{M} \left(f,f_1,\ldots, f_m\right) \right\|_{1} \right\}
\]
\end{proposition}
See \cites{DGW1,DWW} for proofs, and \cite{CDPO1} as well as \cite{DGW2}*{Section 3} for  an account of the weighted norm inequalities ensuing.

\begin{definition}[Paraproducts] \label{def:pp}
Hereafter, \[\bb=\{b_{Q}: Q \in \D\}\] is a sequence of complex numbers. Recalling the basis \eqref{e:basis}, we make the identification of the sequence $\bb$ with the function
\begin{equation}\label{e:seq-id}
\bb\coloneqq \sum_{Q \in \D} |Q|b_Q \varphi_Q.\end{equation}
The $m$-linear operator
 defined by
	\begin{equation}
\label{e:defpp}
 \Pi_{\bb}(f_1\ldots,f_m)(x) \coloneqq \sum_{ Q\in \D}  |Q|  b_Q  \zeta_{Q}(f_1,\ldots,f_m) \beta_Q (x), \qquad x\in \mathbb R^d \end{equation} 
where 
	\[ \left\{\beta_Q \in \Psi^{k+1,{\Subset}}(Q), \,\zeta_Q \in C\Phi^{k+1;\Subset}_m (Q):Q\in \D\right\}\] 
is a generic fixed  family, is called  \textit{paraproduct} with symbol $\bb$. 
  Paraproducts are related to wavelet forms by the  equality 
	\begin{equation}\label{e:lambdap}
	\left\langle \Pi_{\mathfrak{b}}(f_1\ldots,f_m),g\right \rangle=\mathrm{V} (\mathfrak b,g,f_1,\ldots,f_m)	 \end{equation}
where $\mathrm{V}$ is the $(m+2)$-linear wavelet form corresponding to the choices $\phi_Q=\varphi_Q, \psi_Q = \beta_Q \otimes \eta_Q $ in \eqref{e:triebel2}. 
\end{definition}

As is well known, when $\mathfrak{b}=f\in \mathrm{BMO}(\R^d)$,  sparse and weighted Lebesgue space estimates for \eqref{e:defpp} in the full multilinear range may be deduced from Proposition \ref{p:sparseforw}. As a mere example, we point out that
\begin{equation}
	\label{e:lebbase}
\left\|  \Pi_{\mathfrak b}: \prod_{j=1}^m L^{p_j}\to L^r\right\| \lesssim \|\mathfrak b\|_{\mathrm{BMO}(\R^d)}, \qquad 	 1<p_1,\ldots, p_m \leq \infty,\qquad  0< \frac{1}{r} \coloneqq \sum_{j=1}^m \frac{1}{p_j}.  
\end{equation}

The main point of this article is to  obtain  sharp estimates, in terms of requirements on the regularity of the symbol $\mathfrak b$, for the action of   \eqref{e:defpp} on Sobolev spaces. These sharp conditions will be in general much weaker than $\mathrm{BMO}(\R^d)$ membership, will depend on the H\"older exponent tuples $(p_1,\ldots, p_m)$ e.g. in \eqref{e:lebbase},  and will be formulated in terms of Triebel-Lizorkin type norms,  along the lines of \cite{DGW3}*{Section 4} in the linear case. These are the object of the next definition.  
\end{definition}

{

\subsection{Triebel-Lizorkin norms}\label{ss:tl-norms} Sobolev space boundedness of wavelet forms is related to a family of symbol norms generalizing $\mathrm{BMO}(\R^d)$ and described in terms of \textit{intrinsic wavelet coefficients}, defined by the maximal quantity \begin{equation}
\label{e:maxwc} \Psi^{k+1,\Subset}_Q (f) \coloneqq \sup_{\psi\in  \Psi^{k+1,\Subset}_Q } \left| \psi(f)\right|.
\end{equation} Although we only need the wavelet class $\Psi^{k+1,\Subset}_Q$ in this article, we keep the full notation for the sake of  comparison with other works.
Referring to \eqref{e:triebelref}, for  $n\in \mathbb \R$, $n\leq k$,  $1\leq q\leq \infty,$ set
\begin{equation}
\label{e:sfon}
\mathrm{S}^{n}_{q,R} f(x)= \left\Vert \frac{\left|\Psi^{k+1,\Subset}_Z(f)\right|}{\ell(Z)^n} \mathbf{1}_Z(x)\right\Vert_{\ell^q(Z\in \mathcal D(R))}, \qquad R\in \mathcal D.
\end{equation}
The homogeneous  unified Morrey-Campanato-Besov-Triebel-Lizorkin norms we consider are the following. For  $n,m \in \mathbb R, n,m\leq k$ and $1 \le p,q \le \infty$, set
	\begin{align} \label{e:F} \|f\|_{\dot F^{n,m}_{p,q} } &\coloneqq  \sup_{Q \in \mathcal D} \ell(Q)^{-m} \left\langle \mathrm{S}^{n}_{q,Q} f \right\rangle_{p,Q}.  
	\end{align}
	We immediately record the  embeddings
\begin{align}\label{e:Femb}
 \|f\|_{\dot F^{n,m}_{p,q} } & \le \|f\|_{\dot F^{n+u,m-u}_{r,s}(\mathcal R)}, \qquad u \ge 0, \ r \ge p, \ s \le q,
\end{align}
as an immediate consequence of the definition. 
For reference, we point out the norm $ \dot F^{n,m}_{p,q}  $ coincides in essence with the $\dot F^{n,\frac md + \frac 1p}_{p,q}$ norm of \cite{YSY}.
Some clarification on the r\^ole of \eqref{e:F} is provided by the fact that $\mathrm{S}^{n}_{2,R}$ is an instance of the local square function for $|\nabla|^n f$. Indeed,
relying on the $k+1$ vanishing moments of the wavelet, and integrating by parts, standard usage of Littlewood-Paley estimates implies
\begin{equation}
\label{e:lsfo}  
\max\{p,p'\}^{-\frac12} \left \langle \mathrm{S}^{n}_{2,R} f \right\rangle_{p,R}  \lesssim  \inf_{\mathsf{P} \in \mathbb P^{k-n}}  \left\langle \nabla^{n} f - \mathsf P  \right\rangle_{p,\mathsf{w}R} \lesssim 
\max\{p,p'\}^{\frac12}\left\langle \mathrm{S}^{n}_{2,R }f \right\rangle_{p,R}\end{equation}
where $1<p<\infty,$ $ \mathbb P^{m}$ stands for the ring of (vector) polynomials of degree at most $m$, and the implied constants are absolute.
Furthermore, the John-Nirenberg inequality tells us that the norms  $\dot F^{n,0}_{p,2} $, $0< p<\infty$,  are all equivalent, and comparable   with $|\nabla|^n f\in \mathrm{BMO}(\R^d)$.

\subsection{Estimates for localized wavelet forms} The norms in \eqref{e:F} arise in the estimation  of the  {intrinsic wavelet form} localized to some $Q_0\in \mathcal D$, which is defined momentarily in \eqref{e:intrinsicQ_0}.
 For $Q_0 \in \mathcal D$, referring to \eqref{e:triebel2}, we say that the wavelet form $\Lambda$  is \textit{localized} to $Q_0$ if
\[
\psi_Q = 0\quad  \forall Q \not\in  \mathcal D(Q_0), \qquad  \psi_Q \in \Psi^{k+1,\Subset}_m (Q) \quad\forall Q \in \D(Q_0).
\] 
Referring to \eqref{e:maxwc}, each   wavelet form localized to $Q_0$ is dominated by the \emph{intrinsic wavelet form}
\begin{equation}
\label{e:intrinsicQ_0}
\Lambda_{Q_0} (f,f_1,\ldots,f_m) \coloneqq \sum_{Q\in \mathcal D(Q_0)} |Q|\Psi^{k+1,\Subset}_Q (f) \Psi^{k+1,\Subset}_Q (f_1)\prod_{j=2}^m \langle f_j \rangle_{1,wQ}.
\end{equation}
where, as before, $\mathsf{w} \geq 1$ is the dilation parameter specified in the statement of Proposition \ref{prop:triebel}.

\begin{lemma} \label{l:intest} Let $  \theta \in \mathbb Z$  and 
\begin{equation} \label{e:exponents} 
 1< p,q, p_2,\ldots, p_m \leq \infty,\qquad \;{\textstyle\frac{1}{p} +\frac{1}{q} }+ \sum_{j=2}^m{\textstyle\frac{1}{p_m}=1}.
\end{equation}
Then
\[
\Lambda_{Q_0}(b,g,f_2,\ldots, f_m) \lesssim  |Q_0|\ell(Q_0)^{-\theta} \|b\|_{\dot F^{0,-\theta}_{p,2}}  \left\langle   g \right \rangle_{q,\mathsf{w}Q_0}  \prod_{j=2}^m\left\langle  f_j \right \rangle_{p_j,\mathsf{w}Q_0}
\]
with implied constant depending only on $p,p_1,\ldots,p_m$, $d,m$.
\end{lemma}
\begin{proof} Note that we may assume that $f_1,\ldots, f_m$ are all supported in $\mathsf{w}Q_0$. We will prove the more precise estimate
\begin{equation}
	\label{e:sparseest33}
	\begin{split}
		\Lambda_{Q_0}(b,g,f_2,\ldots, f_m) \lesssim \sum_{Q\in \mathcal S} |Q| \langle \mathrm{S}^{0}_{2,Q} b \rangle_{1,Q} \langle \mathrm{S}^{0}_{2,Q} g \rangle_{1,Q} \prod_{j=2}^m \langle f_j \rangle_{1,Q}
	\end{split}
\end{equation}
for some sparse collection $\mathcal S\subset \mathcal D$ with the property that $Q\subset \mathsf{w}Q_0$ for each $Q\in \mathcal S$. We sketch the proof of \eqref{e:sparseest33}. Let $S\in \mathcal S(Q_0)$ be the collection of maximal elements of $\mathcal D(Q_0)$ with the property that at least one of the inequalities
\begin{align} &
\inf_{x\in Q} \mathrm{S}^{0}_{2,Q_0} b > \Theta  \langle \mathrm{S}^{0}_{2,Q_0} b \rangle_{1,Q_0}, \label{e:stop1} \\
&\inf_{x\in Q} \mathrm{S}^{0}_{2,Q_0} g > \Theta  \langle \mathrm{S}^{0}_{2,Q_0} g \rangle_{1,Q_0}, \label{e:stop2}\\&
\inf_{x\in {Q}} \mathrm{M}\left(f_j \mathbf{1}_{\mathsf{w}Q_0}\right) > \Theta  \langle   f_j \rangle_{1,\mathsf{w}Q_0}, \qquad j=2,\ldots, m, \label{e:stop3}
\end{align}
holds. If $\Theta $ is large enough, the packing condition 
\begin{equation} \label{e:packhere2}
	\sum_{Q\in \mathcal S(Q_0) } |S| \leq 2^{-6} |Q_0|
\end{equation}
is easily verified by the maximal theorem. Define \[
\mathcal G(Q_0)\coloneqq \mathcal D(Q_0)\setminus \bigcup_{S\in \mathcal S(Q_0)}  \mathcal D(S). \]  The principal effect of the stopping conditions \eqref{e:stop1}-\eqref{e:stop2} is that the stopped square function
\[
\mathrm{S}^{0,\star}_{2,Q_0} f(x)\coloneqq \left(\sum_{G\in \mathcal G(Q_0)}{\left|\Psi^{k+1,\Subset}_G(f)\right|^2} \mathbf{1}_Z(x)\right)^{\frac12}, \qquad x\in Q_0
\]
satisfies
\[
\left\|
\mathrm{S}^{0,\star}_{2,Q_0} b \right\|_\infty \leq \Theta  \langle \mathrm{S}^{0}_{2,Q_0} b \rangle_{1,Q_0}, \qquad \left\|
\mathrm{S}^{0,\star}_{2,Q_0} g \right\|_\infty \leq \Theta  \langle \mathrm{S}^{-0}_{2,Q_0} g \rangle_{1,Q_0}.
\]
while \eqref{e:stop3} implies 
\[
\sup_{G\in \mathcal G(Q_0)}    \langle   f_j \rangle_{1,\mathsf{w}G} \lesssim    \langle   f_j \rangle_{1,\mathsf{w}Q_0},  \qquad j=2,\ldots,m.
\]
Therefore, using Cauchy-Schwarz to step to the third line,
\[
\begin{split}  &\quad
\Lambda_{Q_0}(b,g,f_2,\ldots, f_m) - \sum_{S\in \mathcal S(Q_0)}\Lambda_{S}(b,g,f_2,\ldots, f_m) \\ & = \sum_{G\in \mathcal{G}(Q_0)}
|G|\Psi^{k+1,\Subset}_G (b) \Psi^{k+1,\Subset}_G (g)\prod_{j=2}^m \langle f_j \rangle_{1,wG}\\ & \lesssim |Q_0| \left\|
\mathrm{S}^{0,\star}_{2,Q_0} b \right\|_\infty \left\|
\mathrm{S}^{0,\star}_{2,Q_0} g \right\|_\infty \prod_{j=2}^m  \sup_{G\in \mathcal G(Q_0)}  \langle f_j \rangle_{1,wG} \\ & \lesssim |Q_0|
| \langle \mathrm{S}^{0}_{2,Q_0} b \rangle_{1,Q_0} \langle \mathrm{S}^{-0}_{2,Q_0} g \rangle_{1,Q_0} \prod_{j=2}^m \langle f_j \rangle_{1,Q_0}
\end{split}
\] and \eqref{e:sparseest33} is proved by iteration and taking advantage of the packing condition \eqref{e:packhere2}.
We turn to deducing the proposition from \eqref{e:sparseest33}. Suppose first that $q<\infty$. In that case an immediate consequence of \eqref{e:sparseest33} and H\"older inequality is 
\[
\begin{split}
\Lambda_{Q_0}(b,g,f_2,\ldots, f_m) &\lesssim \int_{Q_0} \left[ \mathrm{S}^{0}_{2,Q_0} b \right] \left[ \mathrm{S}^{0}_{2,Q_0} g \right]\prod_{j=2}^m \mathrm{M} \left[f_j\mathbf{1}_{\mathsf{w} Q_0}\right] \\ 
& \lesssim |Q_0|\ell(Q_0)^{-\theta} \|b\|_{\dot F^{0,-\theta}_{p,2}}  \left\langle \mathrm{S}^{0}_{2,Q_0} g \right \rangle_{q, Q_0}  \prod_{j=2}^m\left\langle  f_j \right \rangle_{p_j,\mathsf{w}Q_0}
\end{split} 
\]
and the claimed estimate of the proposition is obtained by using  Littlewood-Paley theory. If $q=\infty$, note that
\[
\sup_{Q\in \mathcal D( Q_0)}
\left\langle \mathrm{S}^{0}_{2,Q} g \right \rangle_{1, Q}  \lesssim \|g\mathbf{1}_{\mathsf{w} Q_0}\|_{\mathrm{BMO}(\R^d)} \leq \left\langle   g \right \rangle_{\infty,\mathsf{w}Q_0}
\]
and apply again  H\"older's inequality in the remaining exponents.  \end{proof}

\addtocounter{other}{1} 
\subsection{Anti-Integration by Parts}
We will frequently need to convert wavelet coefficients of functions into wavelet coefficients of their higher order derivatives. For this reason, it will be helpful to subtract off a polynomial $\PM^{k}_Q f$ (of degree $k-1$) suitably adapted to $f$ on $Q$. 

Let ${\theta}$ be a fixed smooth function on $\R^d$ with $\int_{\R^d} {\theta} =1$ and support in the unit cube, and for each $Q \in \D$, set 
\[{\theta}_Q \coloneqq |Q|^{-1}\mathsf{Sy}_Q{\theta}.\] Given a Schwartz function $f$, then we define the Taylor-type polynomial
	\begin{equation}\label{eq:sob-func-rep}	
	\PM^{k}_Q f(x) \coloneqq \begin{cases} 0 & k=0, \\ \displaystyle \sum_{|\alpha| \le k-1} \frac{1}{\alpha ! } \int_Q {\theta}_Q(y) \partial^\alpha f(y) (x-y)^\alpha \, {\mathrm{d} y}
	& k\geq 1.
	\end{cases} \end{equation} 
\begin{lemma}\label{l:taylor}
Let $Q \in \D$, $k \in \mathbb N$ and $\delta>0$. For each $\phi_Q \in \Phi^{k+\delta}(Q)$, there exists $\phi^{-k}_Q \in \Phi^{\delta}(Q)$ such that for any $f$,
	\begin{equation}\label{e:anti-ibp} \phi_Q(f-\PM^k_Q f) = \ell(Q)^{k} \phi^{-k}_Q(\nabla^k f).\end{equation}
Furthermore, let $P\in \mathcal D$ such that $P \subset CQ$. Then, for any $R \in \D$, there exists $\phi^{-k}_{R,Q} \in C\Phi^{0,\Subset}(Q)$ such that for all $f$,
	\begin{equation}\label{e:neighbor} \chi_R(\PM_P^k f - \PM_Q^k f ) =  \ell(Q) \mathfrak{d}(Q,R)^{k-1} \chi^{-k}_{R,Q}(\nabla^k f).\end{equation}
\end{lemma}

\addtocounter{other}{1}
\section{Boundedness of paraproduct operators}\label{sec:pp}
This section contains Sobolev space estimates for the paraproduct operators of \eqref{e:defpp}.
In Theorem  \ref{paraproduct-main} below, we indicate by
\[
\left
\langle
\Pi_{\bb}^{\star, j}(f_1,\ldots, f_m), g \right \rangle \coloneqq 
\left
\langle
\Pi_{\bb}(f_1,\ldots, f_{j-1},g,f_{j+1},\ldots, f_m), f_j \right \rangle
\]
the $j$-th adjoint of $\Pi_\bb$, where $f_j$ and $b$ are the entries interacting with the cancellative components of the associated wavelet form.
\begin{theorem}
	 \label{paraproduct-main}Let $\kappa \in  \mathbb Z \cap [-k,k],$  $n_1,\ldots,n_m \in \{0,\ldots,k\}$ such that $n= n_1 + \cdots + n_m \leq k,$ and 
\[ 1<p_1,\ldots, p_m \leq \infty,\qquad  0< \frac{1}{r} \coloneqq \sum_{j=1}^m \frac{1}{p_j}, \qquad \frac{n}{\pi}\coloneqq \sum_{j=1}^m \frac{n_j}{p_j}.\]
Then for each $\ep>0$, there holds
\begin{equation}
\label{e:para-bound}
\left\|  \Pi_{\mathfrak b}: \prod_{j=1}^m W^{\theta_j n,p_j}\to W^{\kappa,r}\right\| \lesssim_\ep \|\mathfrak b\|_{\dot F^{\kappa,-n}_{\pi+\ep,2}}.
\end{equation}
Furthermore, if $\kappa \ge 0$, then for each $j=1,\ldots,m$,
\begin{align}\label{e:para-adjoint-pos} &\left\Vert \Pi_{\bb}^{\star, j} : \prod_{i=1}^m W^{n_i,p_i} \to W^{\kappa,r} \right\Vert \lesssim_\ep \norm{\bb}_{\dot F^{\kappa-n_j,n_j-n}_{\pi^j_+ +\ep,2}},  \\ & \frac{n-n_j}{\pi^j_+} = \sum_{\substack{i=1 \\i \ne j}}^m \frac{n_i}{p_i} \\
\label{e:para-adjoint-neg} &\left\Vert \Pi_{\bb}^{\star, j} : \prod_{i=1}^m W^{n_i,p_i} \to W^{-\kappa,r} \right\Vert \lesssim_\ep \norm{\bb}_{\dot F^{-n_j,n_j-n-\kappa}_{\pi^j_- +\ep,2}},\\  & \frac{n-(n_j-\kappa)}{\pi^j_-} = \kappa + \sum_{\substack{i=1 \\i \ne j}}^m \frac{n_i-\kappa}{p_i} 
. \end{align}
\end{theorem}

\begin{remark}\label{r:epsilon}
Let us make a few comments about the role of $\ep$ in Proposition \ref{e:para-bound}. First, taking $\ep=0$ in \eqref{e:para-bound} in fact characterizes the restricted strong-type estimates where each $W^{n_j,p}$ for $n_j > 0$ is replaced by the Lorentz-Sobolev space $W^{n_j,(p,1)}$; see \eqref{e:tointerpolate} below. Second, in \eqref{e:para-adjoint-pos}, when $n_i=0$ for each $i \ne j$, not only can $\ep$ be zero, but $\pi^j_+$  can be taken to be one using the John-Nirenberg equivalence 
	\[ \norm{\bb}_{\dot F^{u,0}_{p,2}} \sim \norm{\bb}_{\dot F^{u,0}_{1,2}}, \quad u \in \mathbb R, \, 1 \le p < \infty.\]
Finally, $\ep$ can be taken to be zero and \eqref{e:para-bound} persists in the supercritical case, where $n_ip_i>d$. We refer the interested reader to the end of the proof of Lemma 4.18 in \cite{DGW2}.
\end{remark}

\addtocounter{other}{1}
\subsection{Proof of Theorem  \ref{paraproduct-main}} \label{ss4:mainline}

The proof    requires a few pieces of additional notation. First, fixing $\vec n=(n_1,\ldots, n_m) \in  \mathbb N^m$ and $n=n_1+\cdots+n_m,$  our test class for the norm inequalities  is $f\in \mathcal C_0^{n+1}(\R^d)$. 
The main new object is a version of the adjoint form \eqref{e:lambdap}  localized to $Q\in \mathcal D$, namely
\begin{equation}
\label{e:addnot1}
\begin{split}
\mathrm{V}_{Q}  (\mathfrak b,g,f_1,\ldots,f_m)& \coloneqq \sum_{R\in \mathcal D(Q)} |R|  b_R  \beta_R(g) \zeta_{R}\left(f_1 ,\ldots,  f_m\right)
\end{split}
\end{equation}
\begin{lemma} \label{l:mainiter} Let $p,q,p_2,\ldots,p_m$ be as in \eqref{e:exponents}.  There exist a sparse collection $\mathcal Z(Q)\subset \mathcal D(Q)$ with the property that
\[\begin{split}&\quad \mathrm{V}_{Q}\left(\mathfrak b, g, f_1-\mathsf{P}_Q^n f_1,f_2,\ldots,f_m\right) \\ &\lesssim \|\mathfrak b\|_{\dot F^{0,-n}_{p,2}} \sum_{Z\in \mathcal Z(Q)} |Z| \langle g \rangle_{q,\mathsf{w}Z} \langle \nabla^{n} f_1 \rangle_{1,\mathsf{w}Z}  \prod_{u=2}^m \langle  f_u \rangle_{p_j,\mathsf{w}Z}.
 \end{split}\]
\end{lemma} Lemma \ref{l:mainiter} is proved in \S \ref{ss:mainiter} below. Let us proceed to prove Theorem  \ref{paraproduct-main}. Fix $f_1,\ldots,f_m,g \in \C^{n+1}_0(\mathbb R^d)$. We claim that for any $\ep >0$, there exists a cube $Q_0$ with $\ell(Q_0) \ge 1$ such that
	\begin{equation}\label{e:limit} \left\vert \mathrm V(\bb,g,f_1,\ldots,f_m) - \mathrm V_{Q_0}(\bb,g,f_1-\PM^n_{Q_0} f_1,f_2,\ldots,f_m) \right\vert \le \ep.\end{equation}
Indeed, if $Q \in \mathcal D(Q_0)$ then $\avg{\PM^n_{Q_0}f}_Q \lesssim \sum_{j=0}^{n-1} \avg{\nabla^j f}_{Q_0} \ell(Q_0)^j \lesssim \norm{f}_{W^{n,1}} \ell(Q_0)^{k-1} |Q_0|^{-1}$ so that, by Lemma \ref{l:intest}
	\begin{equation}\label{e:limit-step1} \left\vert \mathrm V_{Q_0}(\bb,g,\PM^n_{Q_0} f_1,f_2,\ldots,f_m) \right\vert \lesssim \ell(Q_0)^{-1} \|\bb\|_{\dot F^{0,-n}_{p,2}}  \left\langle g \right \rangle_{q,\mathsf{w}Q_0}  \prod_{u=2}^m\left\langle  f_u \right \rangle_{p_u,\mathsf{w}Q_0}\end{equation}
which clearly goes to zero as $\ell(Q_0) \to \infty$. Finally, by taking $Q_0$ large enough that $f_1,\ldots,f_m,g$ are all supported on $Q_0$, for any $Q \in \D$ such that $Q_0 \subset Q$, there holds 
	\[ \avg{h}_{\mathsf w Q} \lesssim \frac{|Q_0|}{|Q|} \avg{h}_{Q_0}, \quad h \in \{f_1,\ldots,f_m,g\},\]
which, combined with the trivial estimate $\vert b_Q \vert \lesssim \norm{\bb}_{\dot F^{-n,0}_{1,\infty}} \ell(Q)^{-n}$,
	\begin{equation}\label{e:limit-step2} \begin{split} &\quad \left\vert \sum_{\substack{Q \in \mathcal D, \\ Q_0 \subset Q }} |Q| b_Q \Psi^{n+1,\Subset}_Q(g) \prod_{u=1}^m \avg{f_u}_{\mathsf w Q} \right\vert \\
	& \lesssim \norm{\bb}_{\dot F^{-n,0}_{1,\infty}} \avg{g}_{Q_0} \norm{f}_{L^1} \prod_{u=2}^m \avg{f_u}_{Q_0} \sum_{\substack{Q \in \mathcal D, \\ Q_0 \subset Q }}\ell(Q)^{-n}  \left( \frac{|Q_0|}{|Q|} \right)^{m}.\end{split}\end{equation}
The geometric series clearly goes to $0$ as $\ell(Q_0) \to \infty$. Combining \eqref{e:limit-step1} and \eqref{e:limit-step2} establishes \eqref{e:limit} and thus the following sparse bound holds by virtue of Lemma \ref{l:mainiter}.
\begin{equation}
\label{e:sparseafter} \left\vert \mathrm V ({\mathfrak b},g,f_1,\ldots f_m) \right\vert \lesssim \|b\|_{\dot F^{0,-n}_{p,2}} \left\| \mathrm{M}_{(1,p_2,\ldots,p_m,q)}\left(\nabla^n f_1, f_2, \ldots f_m, g\right) \right\|_1
\end{equation} 
which in particular entails the norm estimate with $p_1=p$ and $r=q'$,
\begin{equation}
\label{e:tointerpolate}
\left\|  \Pi_{\mathfrak b}: W^{n,(p_1,1)}\times L^{p_2}\times\cdots \times   L^{p_m}\to L^r\right\| \lesssim \|\bb\|_{\dot F^{0,-n}_{p_1,2}}. 
\end{equation}
 see e.g. \cite{CDPO2}*{Appendix A}. 
  If one uses strong type estimates, there holds instead
 \begin{equation}
\label{e:tointerpolate2}
\left\|  \Pi_{\mathfrak b}: W^{n,p_1}\times L^{p_2}\times\cdots \times L^{p_m}\to ^L{r}\right\| \lesssim_\varepsilon \|\bb\|_{\dot F^{\kappa,-n}_{p_1+\varepsilon,2}} 
\end{equation}
for $\varepsilon>0$ and it is a bit more manageable to interpolate \eqref{e:tointerpolate2}, although \eqref{e:tointerpolate} may be dealt with as well giving a bit more precise results; see Remark \ref{r:epsilon}.
For instance, \eqref{e:tointerpolate2} can be turned into 
 \begin{equation}
\label{e:interpolated}
\left\|  \Pi_{\mathfrak b}: \prod_{j=1}^m W^{\theta_j n,p_j}\to L^r\right\| \lesssim_\ep \|\bb\|_{\dot F^{0,-n}_{\pi+\ep,2}},
\quad   \theta_j \geq 0,  \; \sum_{j=1}^m \theta_j =1, \; {\textstyle \frac{1}{\pi} }= \sum_{j=1}^m {\textstyle \frac{\theta_j}{p_j}},
\end{equation}
from which we will now derive the three estimates of which Theorem  \ref{paraproduct-main} consists.
In fact, from \eqref{e:interpolated} and integrating by parts we can arrive at the full scale of positive and negative Sobolev space bounds for $\Pi_{\bb}$ and its adjoints. Let $\gamma,\iota,\alpha_j,\beta_j \in \mathbb N$, $j=1,\ldots,n$. Simply by integrating by parts (and in the case $\iota < \gamma$ using \eqref{e:anti-ibp} which relies on the fact that $g$ is paired with an element of $\Psi^{k+1,\Subset}(Q)$),
	\[ \begin{split} &\ip{\Pi_{\bb}(\nabla^{\alpha_1}f_1,\ldots,\nabla^{\alpha_n}f_n)}{\nabla^\iota g} = \mathrm{V}(\tilde \bb,\nabla^{\gamma}g,f_1,f_2,\ldots,f_n), \\ &
	 \tilde \bb = \sum_{Q} b_Q \ell(Q)^{\gamma-\iota-\alpha} \varphi_Q , \quad \alpha = \sum_{j=1}^n \alpha_j. \end{split} \]
Set now $\beta = \sum_{j=1}^n \beta_j$ and let $\pi > 1$ be defined by $\frac{\beta}{\pi} = \sum_{j=1}^n \frac{\beta_j}{p_j}$. With a view towards applying \eqref{e:interpolated} to $\mathrm{V}$ in the above display, notice that
	\[ \norm{\tilde \bb}_{\dot F^{0,-\beta}_{\pi+\ep,2}} = \norm{\bb}_{\dot F^{\iota+\alpha-\gamma,-\beta}_{\pi+\ep,2}}.\]
Therefore applying \eqref{e:interpolated} (take $\theta_j = \frac{\beta_j}{\beta}$), we obtain
	\begin{equation}\label{e:para-full-range} \abs{ \ip{\Pi_{\bb}(\nabla^{\alpha_1}f_1,\ldots,\nabla^{\alpha_n}f_n)}{\nabla^\iota g} } \lesssim \norm{\bb}_{\dot F^{\iota+\alpha-\gamma,-\beta}_{\pi+\ep,2}}\norm{g}_{W^{\gamma,r'}} \prod_{j=1}^m \norm{f_j}_{W^{\beta_j,p_j}} .\end{equation}
Of course, if $\alpha_j$ and $\beta_j$ are both positive for some $j$, then the first step (integrating by parts to move all the $\alpha_j$ derivatives) was a bad idea, so one should actually optimize the choice of $\beta_j$ and $\alpha_j$ before hand, or just assume that for each $j$, $\min\{\alpha_j,\beta_j\}=0$. In our applications \eqref{e:para-bound}, \eqref{e:para-adjoint-pos}, and \eqref{e:para-adjoint-neg}, we will choose the latter option. To prove \eqref{e:para-bound}, take \eqref{e:para-full-range} with
	\[ \alpha_j=0, \quad \beta_j=n_j, \quad \iota=\max\{0,\kappa\}, \quad \gamma = \max\{0,-\kappa\}.\]
Now, to prove the bounds for the adjoints $\Pi^{\star, j}_{\bb}$, first notice that for any $\kappa \ge 0$,
	\[ \ip{\Pi_{\bb}^{\star, j}(f_1,\ldots,f_n)}{\nabla^\kappa g} = \ip{\Pi_{\bb}(f_1,\ldots,f_{j-1},\nabla^\kappa g,f_{j+1},\ldots,f_m)}{f_j}. \]
Now, we may exchange the roles of $r'=q$ and $p_j$ and apply \eqref{e:para-full-range} with
	\[ \alpha_j=\kappa, \quad \beta_j = 0, \quad \alpha_i=0, \quad \beta_i=n_i, \quad \gamma=n_j, \quad \iota=0\] 
to achieve \eqref{e:para-adjoint-pos}. For \eqref{e:para-adjoint-neg}, we do the same thing except take $\alpha_j=0$ and $\beta_j = \kappa$.
\qed 

\subsection{Proof of Lemma \ref{l:mainiter}}\label{ss:mainiter} The proof is iterative, and begins with the following definition. Let $ \mathcal S(Q)$ be the collection of maximal elements  $Z\in \mathcal D(Q)$ with the property that \[
\mathsf{w}Z\subset  \left\{\mathrm{M} \left(\mathbf{1}_{\mathsf{w} Q}\nabla^{n} f_1\right)>C  \langle \nabla^n f_1\rangle_{\mathsf{w}Q}\right\}.\]   Appealing to the maximal theorem, we learn that  
\begin{equation}
\label{e:packinghere}
\sum_{Z\in\mathcal S(Q) } |Z|\leq \frac{|Q|}{4} \end{equation} provided $C$ is chosen sufficiently large.
 The key is the estimation of the difference
\begin{equation}
\label{e:sparsesplit}
\begin{split}&\quad\mathrm{V}_{Q}  (\mathfrak b,g,f_1-\mathsf{P}_Q^n f_1,f_2\ldots,f_m) - \sum_{Z\in \mathcal S(Q) }\mathrm{V}_{Z }  (\mathfrak b,g,f_1-\mathsf{P}_Z^n f_1,f_2,\ldots,f_m)
\\ & =\sum_{G\in \mathcal G(Q)}  b_Q \beta_Q (g) \zeta_{Q}(f_1-\mathsf{P}_Q^n f_1,f_2\ldots,f_m )\\ &\quad +\sum_{Z\in \mathcal S(Q)}\mathrm{V}_{Z}  (\mathfrak b,g, \mathsf{P}_Z^n f_1 -\mathsf{P}_Q^n f_1,f_2\ldots,f_m)
\end{split}
\end{equation}
having  introduced the collection $\mathcal G(Q)\coloneqq  \mathcal D(Q)\setminus \bigcup\{\mathcal D(Z): Z\in \mathcal S(Q)\}.$ 
We estimate the first term in \eqref{e:sparsesplit}. The key is to use the two estimates of Lemma \ref{l:taylor} and telescoping to get that for each $P\in \mathcal D(Q)$
\begin{equation}
\label{e:sparsesplit3}
\langle f_1 -\PM^{n}_{Q}f_1 \rangle_{1,\mathsf{w}P} \lesssim   \ell(Q)^n   \inf_{P} \mathrm{M} (\mathbf{1}_{\mathsf w Q}\nabla^n f_1).
\end{equation}
 Referring to \eqref{e:intrinsicQ_0}, we then have
\begin{equation}
\label{e:sparsesplit2}
\begin{split} &\quad\sum_{G\in \mathcal G(Q) }|b_Q |  | \beta_Q (g)  | \left|\zeta_{Q}(f_1-\mathsf{P}_Q^n f_1,f_2\ldots,f_m ) \right| \\ &\leq \left(\sup_{G\in \mathcal G(Q)  } \langle f_1 -\PM^{n}_{Q}f_1 \rangle_{1,\mathsf{w}G}\right) \Lambda_Q(b,g,f_2,\ldots,f_m) \\ &\lesssim   |Q| \|b\|_{\dot F^{\kappa,-n}_{p,2}}\langle \nabla^n f_1\rangle_{\mathsf{w}Q} \langle |\nabla|^{-\kappa} g \rangle_{q,\mathsf{w}Q}  
 \prod_{u=2}^m \langle  f_u \rangle_{p_j,\mathsf{w}Z}
\end{split}
\end{equation}
having used \eqref{e:sparsesplit3}, the non-stopping nature of $\mathcal G(Q)$, and Lemma \ref{l:intest} to pass to the third line. For the estimation of the second term in \eqref{e:sparsesplit}, take into account the bound  \begin{equation}
\label{e:sparsesplit4}
\langle\PM^{n}_{Z}f_1-\PM^{n}_{Q}f_1 \rangle_{1,\mathsf{w}P} \lesssim   \ell(Q)^n   \inf_{Z} \mathrm{M} (\mathbf{1}_{\mathsf w Q}\nabla^n f_1)
\end{equation}for each $P\in \mathcal D(Z)$ and $Z\in \mathcal S(Q)$,
which is also a consequence of Lemma \ref{l:taylor} and telescoping, and estimate
\begin{equation}
\label{e:sparsesplit5}
\begin{split} &\quad\sum_{Z\in \mathcal S(Q)}\mathrm{V}_{Z}  (\mathfrak b,g, \mathsf{P}_Z^n f_1 -\mathsf{P}_Q^n f_1,f_2\ldots,f_m) \\ & \lesssim  \ell(Q)^n\left( \sup_{Z\in \mathcal S(Q)} \inf_{Z} \mathrm{M} (\mathbf{1}_{\mathsf w Q}\nabla^n f_1)\right)\Lambda_Q(b,g,f_2,\ldots,f_m)
\\ & \lesssim  \ell(Q)^n\langle \nabla^n f_1\rangle_{\mathsf{w}Q}\Lambda_Q(b,g,f_2,\ldots,f_m)  \\ & \lesssim  |Q| \|b\|_{\dot F^{\kappa,-\theta}_{p,2}}\langle \nabla^n f_1\rangle_{\mathsf{w}Q} \langle |\nabla|^{-\kappa} g \rangle_{q,\mathsf{w}Q}  
 \prod_{u=2}^m \langle  f_u \rangle_{p_j,\mathsf{w}Z}.
\end{split}
\end{equation}
using the non-stopping nature of the parent of $Z\in \mathcal S(Q)$, and arguing just as for \eqref{e:sparsesplit2}. We have turned \eqref{e:sparsesplit} into the estimate
\[\begin{split}
\left|\mathrm{V}_{Q}  (\mathfrak b,g,f_1-\mathsf{P}_Q^n f_1,f_2\ldots,f_m) \right| &\leq C |Q| \|b\|_{\dot F^{\kappa,-n}_{p,2}}\langle \nabla^n f_1\rangle_{\mathsf{w}Q} \langle |\nabla|^{-\kappa} g \rangle_{q,\mathsf{w}Q} \prod_{u=2}^m \langle  f_u \rangle_{p_j,\mathsf{w}Z}\\ &\quad +  \sum_{Z\in \mathcal S(Q) }\left|\mathrm{V}_{Z}  (\mathfrak b,g,f_1-\mathsf{P}_Z^n f_1,f_2,\ldots,f_m)\right|
\end{split}
\]
which may be iterated, yielding the sparse collection $\mathcal Z(Q)$ in view of the packing estimate \eqref{e:packinghere}. We omit the well-known details. This completes the proof of the Lemma.

\section{Representation and Sobolev regularity of  multilinear Calder\'on-Zygmund operators}\label{sec:rep}
As anticipated in the introduction,  multilinear singular integrals of Calder\'on-Zygmund type enjoying additional smoothness properties can be represented as a finite linear combination of wavelet operators and paraproducts, extending to the multilinear, non-homogeneous case the analysis initiated by the authors in \cite{DGW1}, see also \cite{DWW} for the linear homogeneous case.  
The estimates of Section \ref{sec:pp} then yield Sobolev-type testing conditions  on the paraproduct symbols occurring in the representation, which should be seen as an extension of the classical result by David and Journ\'e \cite{DJ} to the Sobolev case. Before the statement, let us provide a precise definition for the class of singular integral forms we represent.
\begin{definition}\label{def:si} Let $k \in \mathbb N$, $n\geq 1$, $\delta>0$. We say an $({n}+1)$-linear form $\Lambda$ acting on $(n+1)$-tuples of   is a \textit{normalized $k$-smooth  $(\mathrm{SI}_\delta)$ form}  if the following conditions hold
First, $\Lambda$ satisfies the \textit{weak boundedness property}, 
		\begin{equation}\label{WBP}   |{Q}|^n |\Lambda( \chi_{Q}^0,\chi_{Q}^1,\ldots,\chi_{Q}^n )| \le 1, \quad \forall  \chi_{Q}^j \in \Phi^{k+\delta,\Subset}({Q}), \quad {Q} \in\D.\end{equation}
Second, the \textit{$k+\delta$ kernel estimates} hold for $\Lambda$; that is,  there exists a kernel $K$ on $\R^{(n+1)d}$, locally integrable off the diagonal $\{(y,\ldots,y):y\in \R^{d}\}$ in $\R^{(n+1)d}$, such that for any $f_0,\ldots, f_n \in \mathcal W_0$ with $\cap_{j=0}^n \supp f_j =\varnothing$, there holds
\[  \Lambda( f_0,f_1,\ldots,f_n) = \int_{\mathbb R^{d(n+1)}} K(x) \prod_{j=0}^n f_j(x_j) \, \mathrm{d} x \]
and for all $0 \le j \le k$,
	\[ |\nabla_{x_0}^j K(x)| \le \left(\sum_{i =1}^n|x_0-x_i|\right)^{-nd-j}, \quad x = (x_0,x_1,\ldots,x_n) \in (\mathbb R^d)^{n+1}, \]
	\[ \left(\sum_{i =1}^n|x_0-x_i|\right)^{k} |\nabla_{x_0}^k \Delta^0_h K(x)| + \max_{i=1,\ldots,n} | \Delta^i_h K(x) | \le |h|^\delta \left(\sum_{i=1}^n|x_0-x_i|\right)^{-nd-\delta},\]
where $\Delta_h^i$ denotes the difference in the $i$-th variable, 
	\[ \Delta_h^i K(x) \coloneqq K(x_0,\ldots,x_i+h,\ldots,x_n) - K(x_0,\ldots,x_n).\]

\end{definition}
\begin{remark}\label{rem:symm} 
Let $ \Lambda^{\star, j}(f)$ be the $j$-th adjoint of $\Lambda$, that is
	\[ \Lambda^{\star, j} (f_0,f_1,\ldots,f_n) = \Lambda(f_j,f_1,\ldots,f_{j-1},f_0,f_{j+1}, \ldots,f_n).\]
In general, it is \textit {not} true that  $ \Lambda^{\star, j}$ is a  $k$-smooth  $(\mathrm{SI}_\delta)$ form for $j \ne 0$ whenever $\Lambda$ is. However, in such a case $\Lambda^{\star, j}$ is   $0$-smooth  $(\mathrm{SI}_\delta)$form. See \cite{DGW3}*{Section 5} for details.
\end{remark}

\subsection{Testing conditions} 
Let $\gamma  =(\gamma_1,\ldots, \gamma_n)\in \otimes_{j=1}^n \mathbb N^{d}$ with $
\|\gamma\|_{\ell^1}\le k$. Define the \textit{paraproduct symbols of $\Lambda$ of order $\gamma$} by the sequences 
\begin{equation} \label{def:cz}
	\begin{split}
	&\bb^{\gamma} =\{b^{\gamma}_{Q}:Q \in \mathcal S\},  \quad \bb^{\star,  j} = \{b^{\star,  j}_{Q}:Q \in \mathcal S\};\\
		&b^{\gamma}_{Q} \coloneqq \ell(Q)^k \Lambda\big(\varphi_Q, \Tr_Q \mathsf{x}_1^{\gamma_1}, \ldots, \Tr_Q  \mathsf{x}_n^{\gamma_n}\big),
		 \qquad  \quad |\gamma| \le k ;\\[6pt]
		&b^{\star ,j}_{Q} \coloneqq \ell(Q)^k \Lambda^{\star,j} \big(\varphi_Q,\mathbf 1_1,\ldots \mathbf 1_n\big),\qquad\qquad \quad \,
		   j=1,\ldots,n.
		\end{split}
	\end{equation}
Recall that each $\mathsf{x}_j^{\gamma_j}: \mathbb R^d \to \mathbb R$ is the monomial function $\mathsf{x}_j^{\gamma_j}(y_1,\ldots, y_d)=y_1^{(\gamma_j)_1}\cdots y_d^{(\gamma_j)_d} $. With reference to \eqref{e:seq-id}, form each  function $\bb^{\gamma}$ or $\bb^{\star,  j}$ from the corresponding sequence of wavelet coefficients. It is immediate from the weak boundedness and kernel estimates of a $k$-smooth  $(\mathrm{SI}_\delta)$ form that
	\begin{equation}\label{e:Finfty} \|\bb^{\gamma}\|_{F^{|\gamma|-k,0}_{\infty,\infty}} \lesssim 1, \qquad |\gamma | \leq k\end{equation} 
	cf.\ \cite{DGW2} and estimate \eqref{e:basis}.
However, stronger testing type conditions on the symbols $\bb^{\gamma}$ are needed to ensure Sobolev space bounds for the form $\Lambda$. See \eqref{e:Fn} below.
Since our representation theorem is motivated by Sobolev estimates for the adjoint operators, for any $(n+1)$-linear form $\Lambda$, we introduce the vector form
	\[ \nabla^k \Lambda (f_0,f_1,\ldots,f_n) \coloneqq \left( \Lambda(\partial^\alpha f_0,f_1,\ldots,f_n) : |\alpha|=k \right).\]

\begin{theorem}\label{thm:rep}
Let $\Lambda$ be a normalized $k$-smooth $(n+1)$-linear SI form. There exists wavelet forms $\{\mathrm{V}_j\}_{j=1}^{m}$ such that for all $f_j \in \mathcal S(\mathbb R^d)$,
\[ \nabla^k \Lambda(f) = \sum_{j=1}^m \left[ \nabla^k \mathrm V_j^{\star,  j}(f) + \nabla^k \Pi_{\bb^{\star,  j}}^{\star,  j} (f) \right]+ \sum_{|\gamma| \le k} \Pi_{\bb^{\gamma}}(f_0,\partial^{\gamma_1}f_1,\ldots,\partial^{\gamma_n}f_n), \]
\end{theorem}

To obtain Sobolev bounds from Theorem \ref{thm:rep}, we introduce the following norms associated to the paraproduct symbols. 
For $1 \le p \le q \le \infty$, define
\begin{equation}
 \label{e:Fn}\begin{aligned}
\norm{\Lambda}_{\dot F(k,p,q)} &\coloneqq  \sup_{ |\gamma|<k-\lfloor \frac dp \rfloor  } \left\|\bb^{\gamma}\right\|_{ \dot F^{0,|\gamma|-k}_{p,2}} +\sup_{k-\lfloor \frac dp \rfloor \leq |\gamma|\leq k-1} \left\|\bb^{\gamma}\right\|_{ \dot F^{0,|\gamma|-k}_{q,2} }, \\
	&\quad + \sup_{|\gamma|=k} \left\|\bb^{\gamma}\right\|_{{\dot F^{0,0}_{1,2}} } + \sup_{j=1,\ldots,n} \left\|\bb^{\star,  j}\right\|_{\dot F^{-k,0}_{1,2}}
\end{aligned}\end{equation}

\begin{cor}\label{c:sob}
Let $1 < p,p_1,\ldots,p_n \le \infty$ satisfy $\frac{1}{p} = \sum_{j=1}^n \frac{1}{p_j}$. Then, for any $q > p$,
	\[ \norm{T(f_1,\ldots,f_n)}_{W^{k,p}} \lesssim \left( 1 + \norm{\Lambda}_{\dot F(k,p,q)} \right)\sum_{\substack{\beta \in \mathbb N^n \\ |\beta|=k}} \prod_{j=1}^n \norm{f_j}_{W^{\beta_j,p_j}}.\]
\end{cor}

\begin{remark} By placing stronger, $p$-independent, testing conditions on the symbols $\bb^\gamma$, one can in fact take $p<1$, and obtain the full range of multilinear weighted estimates. In fact, this was the approach taken in \cite{DGW1}, while our focus in this paper was to provide weaker $p$-dependent conditions to ensure Sobolev boundedness.
\end{remark}

\subsection{Outline of proof of representation Theorem}
Setting
	\[ \mathcal D^+_n(Q) = \left\{ P=(P_1,\ldots,P_n) \in \times_{i=1}^n \mathcal D : \ell(P_1) = \ldots = \ell(P_n) \ge \ell(Q)\right\},\]
the precursors to wavelet forms are
	\[ \begin{aligned} \Sigma_0(f) &= \sum_{Q \in \D} |Q| \phi_Q(\nabla^k f_0) \sum_{P \in \mathcal D^+_n(Q)} \Lambda(\phi_Q,\psi^1_{P_1},\psi^2_{P_2},\ldots,\psi^n_{P_n}) \prod_{i=1}^n |P_i|\psi_{P_i}^i(f_i) , \\
	\Sigma_j(f) &= \sum_{Q \in \D} |Q| \phi_Q(f_0) \sum_{P \in \mathcal D^+_n(Q)} \Lambda(\phi_Q,\psi^1_{P_1},\psi^2_{P_2},\ldots,\psi^n_{P_n}) |P_j| \psi_{P_j}^j(\nabla^k f_j) \prod_{\substack{i=1 \\ i \ne j}}^n |P_i|\psi_{P_i}^i(f_i) , \\
	&\phi_Q,\psi^1_Q \in \Psi^{k+\delta,\Subset}(Q), \quad \psi_P^j \in \Phi^{k+\delta,\Subset}(P).\end{aligned} \]
The main step in the representation theorem is
\begin{proposition}\label{p:step} There exists wavelet forms $\mathrm{V}_j$ such that
	\[ \Sigma_0(f) = \mathrm{V}_0(f_0,\nabla^k f_1,f_2,\ldots,f_m) + \sum_{|\gamma| \le k}\Pi_{\mathfrak{b}^\gamma}(f_0,\partial^{\gamma_1} f_1,\ldots,\partial^{\gamma_n}f_n).\]
and for $j=1,\ldots,n$,
	\[ \Sigma_j(f) = \mathrm{V}_j(\nabla^k f_0,f_1,\ldots,f_n) + \Pi^{\star,j}_{\mathfrak b^{\star,  j}}(f).\]
\end{proposition}
Following the argument in \cite{DGW1}*{p. 80}, Theorem \ref{thm:rep} follows from Proposition \ref{p:step} and the wavelet resolution in Proposition \ref{prop:triebel}.

\begin{proof}[Proof of Proposition \ref{p:step}]
The major difference between the multilinear and linear setting of \cites{DWW,DGW2} is the method of subtracting off the Taylor polynomial in the range
	\[ A(Q) = \{ P \in \mathcal D^+_n(Q): \ell(P_i) \ge 3 \mathsf{w}\ell(Q), \ \max_i \{|c(P_i)-c(Q)|\} \le 3 \mathsf{w} \ell(P_i)\}.\]
Introduce the Taylor polynomials
	\[ \mathsf{T}^j_Q f = \PM^{j+1}_Q f, \quad \dot{\mathsf{T}}^j_Q f = \mathsf{T}^j_Qf - \mathsf{T}^{j-1}_Q f.\]
In fact, the method we propose here is more efficient than \cite{DGW1}.\footnote{Notice that in \cite{DGW1}, one actually subtracts off paraproducts for each $|\gamma_i| \le k$.} For simplicity, let us fix $Q$ and $P$, thus omitting the dependence on them in what follows. Introduce the shorthand
	\[ \tau^{j}_i = \psi^i_{P_i}- \mathsf{T}^j_Q \psi^i_{P_i}, \quad \mathsf{T}^j_i = \mathsf{T}^j_Q \psi^i_{P_i}, \quad \dot{\mathsf{T}}^j_i = \dot{\mathsf{T}}^j_Q \psi^i_{P_i}. \]
Expand
	\begin{equation}\label{e:split1} \Lambda(\phi_Q,\psi^1_{P_1},\ldots,\psi^n_{P_n}) = \Lambda(\phi_Q,\tau^{k}_1,\psi^2_{P_2},\ldots,\psi^n_{P_n}) + \Lambda(\phi_Q,\mathsf{T}^k_1,\psi^2_{P_2},\ldots,\psi^n_{P_n}).\end{equation}
Leave alone the first term and for the second term, expand
	\[\begin{aligned} &\Lambda(\phi_Q,\mathsf{T}^k_1,\psi^2_{P_2},\psi^3_{P_3},\ldots,\psi^n_{P_n}) = \sum_{j=0}^k \Lambda(\phi_Q,\dot {\mathsf{T}}^j_1,\psi^2_{P_2},\psi^3_{P_3}\ldots,\psi^n_{P_n}) \\
	&= \sum_{j=0}^k \Lambda(\phi_Q,\dot {\mathsf{T}}^j_1,\tau^{k-j}_2,\psi^3_{P_3},\ldots,\psi^n_P) + \sum_{j=0}^k \Lambda(\phi_Q,\dot {\mathsf{T}}^j_1,\mathsf{T}^{k-j}_2,\psi^3_{P_3},\ldots,\psi^n_P). \end{aligned} \]
Continuing this process, we obtain $\Lambda(\phi_Q,\psi^1_{P_1},\ldots,\psi^n_{P_n}) = A + B$ where $A$ is finite sum of terms of the form
	\begin{equation}\label{e:A} \Lambda(\phi_Q, \dot{\mathsf{T}}^{\beta_1}_1,\ldots,\dot{\mathsf{T}}^{\beta_j}_j,\tau^{\beta_{j+1}}_{j+1},\psi^{j+2}_{P_{j+2}},\ldots,\psi^{n}_{P_n}), \ j \le n-1, \ \sum_{i=1}^{j+1} \beta_i = k.\end{equation}
and
	\[ B = \sum_{\substack{\beta \in \mathbb N^n \\ |\beta| \le k}} \Lambda(\phi_Q,\dot{\mathsf{T}}^{\beta_1}_1,\ldots,\dot{\mathsf{T}}^{\beta_n}_n).\]
The argument in \cite{DGW1}*{Lemma 3.3} shows that each term of the form \eqref{e:A} is controlled by $(P,Q)_{k+\delta}$ and hence one may integrate by parts, apply wavelet averaging \cite{DGW1}*{Lemma 2.4}, and convert that portion of $\Sigma_0$ into a wavelet form. We claim that $B$ can be converted into paraproducts by the same reasoning as \cite{DGW2}*{Proof of Theorem A, pp. 48-49}.

Handling $\Sigma_j$ for $j \ne 0$ is much easier and only requires subtracting off $\mathsf{T}^0$ in which case the telescoping argument is not needed. This splits $\Lambda(\phi_Q,\psi_P)$ into a term which is controlled by $(P,Q)_\delta$ and the zeroth order paraproduct. In this case, the $(P,Q)_\delta$ decay suffices for wavelet averaging after integrating by parts since $\ell(P) \ge \ell(Q)$. \end{proof}

\bibliography{}
\bibliographystyle{amsplain}

\end{document}